\documentclass[12pt]{amsart}
\usepackage{fullpage}
\usepackage[utf8]{inputenc}
\usepackage[english]{babel}
\usepackage{amsthm}
\usepackage{amsmath}
\usepackage{amssymb}
\usepackage{multirow}
\usepackage{dirtytalk}
\usepackage{hyperref}

\newtheorem{theorem}{Theorem}[section]

\newtheorem{lemma}[theorem]{Lemma}
\newtheorem{proposition}[theorem]{Proposition}
\newtheorem{remark}{Remark}
\newtheorem{definition}{Definition}[section]
\newtheorem{example}{Example}[section]

\newcommand{\QQ}{\mathbb{Q}}
\newcommand{\RR}{\mathbb{R}}
\newcommand{\ZZ}{\mathbb{Z}}
\newcommand{\CC}{\mathbb{C}}
\newcommand{\GG}{\mathbb{G}}
\newcommand{\UU}{\mathbb{U}}
\newcommand{\HH}{\mathbb{H}}
\newcommand{\PP}{\mathbb{P}}

\DeclareMathOperator{\Nm}{Nm}
\DeclareMathOperator{\Id}{Id}
\DeclareMathOperator{\Norm}{Norm}
\DeclareMathOperator{\Cor}{Cor}
\DeclareMathOperator{\Res}{Res}
\DeclareMathOperator{\End}{End}
\DeclareMathOperator{\Tr}{Tr}

\DeclareMathOperator{\disc}{disc}

\DeclareMathOperator{\Sym}{Sym}

\DeclareMathOperator{\Br}{Br}
\DeclareMathOperator{\GL}{GL}
\DeclareMathOperator{\SL}{SL}
\DeclareMathOperator{\SU}{SU}
\DeclareMathOperator{\Sp}{Sp}
\DeclareMathOperator{\SO}{SO}
\DeclareMathOperator{\KS}{KS}
\DeclareMathOperator{\MT}{MT}

\newcommand{\Nmm}{\widetilde{\Nm}}

\title{Hodge structure of K3 type with real multiplication and Simple Abelian Fourfolds with Definite Quaternionic Multiplication}
\author{Yuwei Zhu}

\begin{document}
\maketitle

\begin{abstract} In this paper we give a general construction of transcendental lattices for K3 surfaces with real multiplication by arbitrary field up to degree 6 along with formula for their discriminants. We also show that all simple Abelian fourfolds with definite quaternionic multiplication can be realized as Kuga-Satake varieties of K3 surfaces with Picard rank 16 and real multiplication by a quadratic field by keeping track of the arithmetic input on both sides.
\end{abstract}

\section{Introduction}

In \cite{Zar83} Zarhin has shown that the transcendental lattice of K3 surfaces are sometimes endowed with real or complex multiplication that are by definition Hodge endomorphisms. Since then, some progress has been made towards understanding such surfaces: van Geemen in \cite{Gee06} constructed transcendental lattices of rank 6 with real multiplication by a specific kind of real quadratic fields; Schlickewei in \cite{Sch09} computed the endomorphism ring of the associated Kuga-Satake varieties for general cases by a process that is not canonical; and Elsenhans and Jahnel (see \cite{E-J14}, \cite{E-J16} and \cite{E-J18}) used $L$-functions and point counting to write down specific formulae of K3 surfaces with quadratic RM. Using the language of Mumford-Tate group and generalizing Mumford's construction in \cite{Mum69}, we give a slightly different account of the resulting Kuga-Satake variety for K3 surfaces satisfying certain conditions on the rank of the transcendental lattice, which is a canonical construction. We shall also give a general formula for constructing the transcendental lattices of K3 surfaces with RM by arbitrary totally real field up to degree 6 over $\QQ$, which will allow us to compute the discriminant form of such lattices. Moreover, we will use these new results to show that all simple Abelian fourfolds with definite quaternionic multiplication are Kuga-Satake varieties of K3 surfaces with quadratic real multiplication, Picard rank 16, and provide a dictionary between the two:

\begin{theorem}\label{main} Any simple Abelian fourfold $A$ with endomorphism by a definite quaternionic order can be realized as a Kuga-Satake variety of a K3 surface $S$ of Picard rank $16$ with real multiplication by a quadratic field: if we let the polarization of $A$ to be given by $$\left (
\begin{array}{cc}
0 & D \\
-D & 0 \\
\end{array} \right )$$ then the Mumford-Tate group of $A$ is given by $SO(D)$ for a positive definite quadratic form $D$. Such $D$'s splits into a product of two conjugating conics $Q \times \overline{Q}$ up to a real quadratic field extension $K/\QQ$, with exactly one of which having real points, and the transcendental lattice of the associated K3 surface will have real multiplication by $K$ and its transcendental lattice obtained by the corestriction (to be introduced in section~\ref{sec4}) of $Q$. In particular, the discriminant of the transcendental lattice is $$\disc(T_S)= \disc(K/\QQ)^2 \cdot \Norm_{K/\QQ}(\disc(Q)).$$

Conversely, if $S$ is a Picard rank 16 K3 surface with real multiplication by a quadratic field $K$, then the transcendental lattice $T(S)$ is given by the corestriction of some $K$-ternary quadratic field $(V, Q)$ with signature (defined in~\ref{sec4}) $(2,1), (0,3)$. Up to similarity, $Q$ corresponds to an unique quaternion algebra $B$ defined over $K$ up to isomorphism, and the resulting simple factor of the Kuga-Satake variety of $S$ has rational quaternionic multiplication by $B_{\Cor}$, which represents the same class as $\Cor_{K/\QQ}(B)$ in $\Br(\QQ)$.
\end{theorem}

This paper is organized as follows. In section \ref{sec2}, we recall definitions and some general facts about Mumford-Tate groups, K3 surfaces and Kuga-Satake construction. In section \ref{sec3}, we provide a generalized Mumford's construction that realizes the norm-1 group of some quaternion algebras as Mumford-Tate group of some Abelian varieties with quaternionic multiplication. In section \ref{sec4}, we give a general construction of lattices of K3 type with real multiplication by arbitrary RM field. In section \ref{sec5}, we use the results in section \ref{sec3} and \ref{sec4} to show that all simple Abelian fourfolds with quaternionic multiplication can be realized as the Kuga-Satake variety of some RM K3 surface, and in fact we have fairly nice control over the arithmetic inputs and outputs of this correspondence.

\renewcommand{\abstractname}{Acknowledgments}

\begin{abstract} The author would like to thank Prof. Bert van Geemen for numerous comments, clarifications and suggestions on this paper, Prof. Brendan Hassett for pointing to the subject of K3 surfaces and their Mumford-Tate groups, Prof. Yuri Zarhin for the references on the classification of simple Abelian fourfolds, Prof. John Voight on arithmetic properties of Clifford algebras, and Prof. Frank Calegari for pointing to an error in the draft. The author is supported by the NSF grants DMS-1551514 and DMS-1701659.
\end{abstract}

\section{Background}\label{sec2}

In this section we introduce the definition of Mumford-Tate group and Kuga-Satake construction associated to the transcendental lattice of a K3 surface. 

Recall that a $\QQ$-Hodge structure of weight $m$ is a finite dimensional $\QQ$-vector space $V$ together with a decomposition $V_{\CC} = \bigoplus_{p+q=m} V_{\CC}^{p,q}$ such that $\overline{V_{\CC}^{p,q}} = V_{\CC}^{q,p}$. We can define the action of the Deligne torus $\mathbb{S} := \Res_{\CC/\RR}\GG_{m, \CC}$ on $V_{\CC}^{p,q}$ to be multiplication by $z^{-p}\overline{{z}^{-q}}$. This is equivalent to giving a representation $h: \mathbb{S} \rightarrow \GL(V_{\RR})$. (See 1.1, 3.2, 3.3 of \cite{Moo04}) In particular, one can restrict the map $h$ to the unit circle $\UU = \{ z | z\overline{z}=1 \}$.

\begin{definition} (4.1 of \cite{Moo04}) Let $V$ be a $\QQ$-Hodge structure given by the homomorphism $h: \mathbb{S} \rightarrow \GL(V)_{\RR}$. We define the (general) Mumford-Tate group of $V$, notation $\MT(V)$, to be the smallest algebraic $\QQ$-subgroup $M$ in $\GL(V)$ such that $h$ factors through $M_{\RR} \in \GL(V)_{\RR}$.

Similarly, we define the special Mumford-Tate group to be the smallest algebraic $\QQ$-subgroup such that the restriction of $h$ onto $\UU$ factors through over $\RR$. 

(\cite{Zar83} 0.3.1.1) An element $g$ in $\End(V)$ is called a Hodge endomorphism if it commutes with the action of of the Mumford-Tate group.
\end{definition}

\begin{remark} In some literature (e.g. \cite{Zar83}), the special Mumford-Tate group is called the Hodge group.
\end{remark}

One key property of the Mumford-Tate group is that it translates the problem of locating sub-Hodge structures in tensorial constructions of a $\QQ$-Hodge structure into the problem of locating irreducible $\QQ$-subrepresentations:

\begin{proposition} (4.4 of \cite{Moo04}) Let $W \subset V^{\otimes a} \otimes (V^{\vee} )^ {\otimes b}$ be a $\QQ$-subspace. Then $W$ is a sub-Hodge structure if and only if $W$ is stable under the action of $\MT(V)$ on $V^{\otimes a} \otimes (V^{\vee} )^ {\otimes b}$.

In particular, $W$ corresponds to the $\QQ$-span of a Hodge class if and only if $W$ is 1-dimensional trivial representation of $\MT(V)$.
\end{proposition}

\begin{example} Consider the case when $m=1$ and the $\QQ$-Hodge structure $V$ describes an Abelian variety $A$ up to isogeny. The special Mumford-Tate group for $m=1$ always lies inside $\Sp(V, \psi)$, since it has to preserve a symplectic form that is the polarization $\psi$ of the Abelian variety. Then one can prove that the endomorphism algebra of $A$ as a vector space lies inside the space of its Hodge endomorphisms. Indeed, an endomorphism (upon tensoring with $\QQ$) on $A$ is given by a correspondence, i.e. a $(1,1)$-class in $H^2(A \times A)$ which then can be realized as a $1$-dimensional invariant subspace in $V^{\otimes 2}$. Since the action of special Mumford-Tate group on $V^{\otimes 2}$ is given by conjugation, $g \in V^{\otimes 2} \cong \End_{\QQ}(V)$ commutes with $\gamma \in \MT(V)$ if and only if $\gamma g \gamma^{-1} = g$, i.e. $g$ spans a trivial representation of $\MT(V)$.
\end{example}

In \cite{Zar83} Zarhin has studied extensively the Mumford-Tate group and its actions associated to a K3 surface. Since the Picard group of K3 surfaces are spanned by algebraic cycles (also known to be Hodge cycles), the Mumford-Tate group acts trivially on the Picard group. We recall the definition of K3 surfaces and some facts about its transcendental lattice:

\begin{definition}(\cite{Huy16} Chapter 1, Definition 1.1) A K3 surface over a field $k$ is a complete nonsingular variety $S$ of dimension 2 such that $$\Omega^2_{S/k} \cong \mathcal{O}_S \text{ and } H^1(S, \mathcal{O}_S) = 0.$$

(\cite{Huy16} Chapter 1, Proposition 3.5) The integral cohomology $H^2(S, \ZZ)$ of a complex K3 surface $S$ endowed with the intersection form is abstractly isomorphic to the lattice $E_8(-1)^{\oplus 2} \oplus U^{\oplus 3}$. The Picard lattice embeds primitively into this lattice, its orthogonal complement is called the transcendental lattice of $S$, denoted $T(S)$. $T(S)$ is known to have signature $(2, 20-\rho)$ where $\rho$ is the rank of the Picard group.
\end{definition}

We recall some key properties of $T(S)$ and its Mumford-Tate group proved in \cite{Zar83}:

\begin{proposition} Let $S$ be a complex K3 surface, $T(S)$ its transcendental lattice. Let $G$ be the Mumford-Tate group of $T(S)_{\QQ}$ as a weight 2 $\QQ$-Hodge structure. Then the following is true:

\begin{enumerate}
\item (1.4.1 of \cite{Zar83}) $T(S)_{\QQ}$ is an irreducible representation of $G$.

\item (1.6 of \cite{Zar83}) The ring of Hodge endomorphisms of $T(S)$ tensored with $\QQ$ is either a totally real field or a CM field. We say the K3 surface $S$ has RM or CM respectively.
\item (2.1 and 2.2.1 of \cite{Zar83}) In the case when $S$ has RM by a totally real field $K$, the intersection form on $T(S)_{\QQ}$ is the trace of a $K$-form $Q$, and the special Mumford-Tate group is the special orthogonal group $\SO_K(Q)$ viewed as a $\QQ$ group.
\end{enumerate}
\end{proposition}

We will explain the $K$-form $Q$ in greater detail in section~\ref{sec4} of this paper. The last ingredient we need is the Kuga-Satake construction of a K3 surface $S$, which returns for each K3 surface $S$ an Abelian variety $\KS(T(S))$. A good reference is given in chapter 4 of \cite{Huy16}. We will not go into the detail of this definition, but we will recall two properties of this construction:

\begin{lemma}~\label{ksrep} (\cite{Huy16} Chapter 4, 2.6) There is an inclusion of Hodge structures: $T(S)_{\QQ} \subset H^2(\KS(T(S))^2, \QQ)$. Moreover, the Mumford-Tate groups of $T(S)_{\QQ}$ and $H^1(\KS(T(S)))$ share the same Lie algebra. 

(\cite{Gee06} Lemma 5.5) In the case when $S$ has real multiplication by a field of degree $d$ and the $K$-rank of the transcendental space $n$, over $\CC$ the Mumford-Tate group is given by $SO(n)^d$, and its representation on $H^1(\KS(T(S)), \CC)$ sums of exterior products of the spin representation of $SO(n)^d$.
\end{lemma}

\section{A generalization of Mumford's construction}\label{sec3}

In this section we generalize the construction of Abelian varieties given by Mumford in \cite{Mum69}.

Throughout the paper let $K$ be a totally real number field, let $d \geq 2$ be its degree. Let $B$ be a quaternion algebra defined over $K$ such that\begin{equation}\label{ram} B \otimes _{\QQ} \RR \cong \HH^{d-1} \oplus M_2(\RR) \end{equation} where $\HH$ stands for the Hamiltonian. It is a general fact that there exists field extensions $K'$ of $K$ such that $B \hookrightarrow B \otimes_{K} K' \cong M_2(K')$, hence we can define the reduced norm of an element in $B$ to be its determinant up to this embedding, and reduced trace its trace. The collection of reduced norm 1 elements form a group, usuall called the norm-1 group of the quaternion $B$.

Recall that in \cite{Mum69} Mumford wrote down the explicit expression of $B$ under Serre's corestriction functor: Let $B^{(i)} = B \otimes_K K^{(i)}$ be the quaternion algebra over $\overline{\QQ}$. Tensoring them together we get a central simple algebra $D = B^{(1)} \otimes B^{(2)} \otimes ... \otimes B^{(d)}$. The Galois group $Gal(\overline{\QQ}/\QQ)$ has a natural semilinear action on $D$ given by the following: if $\tau: \overline{\QQ} \rightarrow \overline{\QQ}$ is an element in $Gal(\overline{\QQ}/\QQ)$, then $\tau \circ \sigma_i = \sigma_{\pi(i)}$ for some permutation $\pi$. Similarly $\tau$ acts on $\overline{\QQ}$, which then gives a semilinear isomorphism $B^{(i)} \xrightarrow{\sim} B^{\pi(i)}$, thus a semilinear automorphism on $D$. We define the elements fixed by such Galois actions in $D$ to be $\Cor_{K/\mathbb{Q}}(B)$. It is now a central simple algebra defined over $\mathbb{Q}$. Notice that at each local place $p$ of $\QQ$, the class of $\Cor_{K/\QQ}(B)$ in $\Br(\QQ_p)$ is equal to the sum of all classes $B_{\pi} \in \Br(K_{\pi})$, where $\pi | p$. This implies that $\Cor_{K/\QQ}(B)$ is a 2-torsion class in $\Br(\QQ)$, and is Brauer equivalent to a quaternion algebra $B_{\Cor}$ defined over $\QQ$.

\begin{remark} It is not hard to see that the condition we imposed on $B$ would force $B_{\Cor}$ to be indefinite when $d$ is odd, and definite when $d$ is even.
\end{remark}

We generalize Mumford's construction as follows: consider the central simple algebra $B_{\Cor} \otimes_{\QQ} \Cor_{K/\QQ}(B)$. It is now of dimension $4^{d+1}$; moreover, $B_{\Cor} \otimes_{\QQ} \Cor_{K/\QQ}(B) \cong M_{2^{d+1}}(\QQ)$ exactly because its Brauer classes at all local places add up to 0. Now consider the group of norm-1 elements $G$ in $B$. We define the augmented Nm map to be the following:
$$\begin{array}{cccc}
\widetilde{\Nm}: & B & \rightarrow & B_{\Cor} \otimes \Cor_{K/\QQ}(B) \cong M_{2^{d+1}}(\QQ) \\
 & \gamma & \mapsto & \Id \otimes \gamma^{(1)} \otimes \gamma^{(2)} \otimes ... \otimes \gamma^{(d)} \\
\end{array}$$

Therefore $\widetilde{\Nm}$ is a faithful $\QQ$-representation of $G$ into $\GL_{2^{d+1}}(\QQ)$. Hence it acts on $\QQ^{2^{d+1}}$. To see that it can be in fact realized as a special Mumford-Tate group of an Abelian variety of dimension $2^d$, we show that the image of $\widetilde{\Nm} \otimes \CC$ lies inside $Sp_{2^{d+1}}(\CC)$ and give an embedding of Deligne's torus. 

\begin{lemma} Over $\RR$, $G_{\RR} \cong \SU_2(\RR)^{d-1} \times \SL_2(\RR)$, and it preserves a symplectic form. 
\end{lemma}

\begin{proof} By the definition of $\Nmm$, over $\CC$, $\Nmm$ is isomorphic to the direct sum of two copies of $W_{1,...,1}$, i.e. the exterior product of $d$-many standard representations of $SL_2(\CC)$: $\CC^2 \boxtimes \CC^2 \boxtimes ... \boxtimes \CC^2$. One can compute the decomposition of the exterior product of this representation:$$\wedge^2 (W_{1,...,1} \oplus W_{1, ..., 1}) = 3\wedge^2 W_{1,...,1} \oplus \Sym^2 W_{1,...,1} $$
And $$\Sym^2 W_{1,...,1} = W_{2,2,...,2} \oplus \bigoplus_{\text{even number of 0's in the weight}} W_{2,...,0,...,2} $$
$$\wedge^2 W_{1,...,1} = W_{2,2,...,0} \oplus ... \oplus W_{0,2,...,2} \oplus \bigoplus_{\text{odd number of 0's in the weight}} W_{2,...,0,...,2} $$
When $d$ is even (resp. odd), $\Sym^2 W_{1,...,1}$ (resp. $\wedge^2 W_{1,...,1}$) will have a copy of $W_{0,0,...,0}$ i.e. the trivial representation. Hence $G_{\RR}$ preserves a symplectic form.
\end{proof}

Analogous to Mumford's original construction, we can embed the torus $\Res_{\CC/\RR}(\CC^*)$ via the map: $$h: e^{i\theta} \rightarrow  \Id_{2^d} \otimes \left ( \begin{array}{cc}
cos\theta & sin\theta \\
-sin\theta & cos\theta   \end{array} \right)$$

Notice that $G$ is a $\QQ$-simple group, and the above construction gives abelian varieties $A$ up to isogeny with their Mumford-Tate group contained in $G$. This implies that the Mumford-Tate group of $A$ is exactly $G$. Moreover, we can explicitly describe the endomorphism algebra of $A$:

\begin{theorem}Let $(K, B)$ be given. Then the endomorphism algebra $\End(A) \otimes \QQ$ of a generic member $A$ of the resulting Abelian varieties is exactly $B_{\Cor}$, i.e. $A$ has quaternionic multiplications. 
\end{theorem}

\begin{proof} We first describe how $B_{\Cor}$ acts as Hodge endomorphism. Recall that Hodge endomorphisms are by definition elements in $\End(H^1(A,\QQ))$ that are commutative with the action of Mumford-Tate group, and that upon tensoring with $\QQ$ the algebra of Hodge endomorphisms and endomorphisms of Abelian varieties coincide.

Note that for an element $g \in B_{\Cor}$, there is a faithful embedding into $M_{2^{d+1}}(\QQ) \cong \End(H^1(A,\QQ))$: $g \mapsto g \otimes \Id \otimes \Id \otimes ... \otimes \Id$ where $\Id$ is the identity in each $B^{(i)}$. It is obvious that $g$ commutes with the action of $\Nmm(\gamma)$ for any $\gamma \in G$.

To see that $B_{\Cor}$ indeed gives everything, we notice that when $d$ is odd (resp. even), $H^2(A, \QQ)$ has $3$ (resp. $1$) copies of Hodge cycles, which are the counts of Abelian varieties having either endomorphisms by an indefinite quaternion over $\QQ$ (resp. definite quaternion over $\QQ$) or RM by a cubic field (resp. no extra endomorphism or CM by an imaginary quadratic field). Since the simple factors of $A$ always have dimension $2^{d-1}$ or $2^{d}$ (resp. the Lie algebra of the special Mumford-Tate group is a Lie algebra over a totally real number field), the latter case is not possible by dimension count (resp. by general facts about special Mumford-Tate groups established in \cite{Z-M95} section 2.6). 
\end{proof}

\begin{remark} It is possible for $A$ to split. For example, when $K$ is a cubic field and $\Cor_{K/\QQ}(B)$ is already $M_8(\QQ)$, then the endomorphism algebra of the resulting $A$ is $M_2(\QQ)$, i.e. it is the product of two isogenous Abelian fourfolds. This is exactly the fourfolds Mumford constructed in his original paper (cf. \cite{Mum69}), and the reason we are calling the construction introduced in this section \say{generalized Mumford construction}. 
\end{remark}

\section{Constructing Hodge structure of K3 type with real multiplication}\label{sec4}

We recall the description of the transcendental lattice of a K3 surface with real multiplication structure:

\begin{proposition}
(2.1 in \cite{Zar83}) Let $K$ be a degree $d$ totally real field, $V$ an $n$-dimensional $K$-vector space equipped with a non-degenerate symmetric bilinear form $Q$. A Hodge structure of K3 type with real multiplication by $K$ is a vector space $V_0$ defined over $\QQ$ with the following extra structure:
\begin{enumerate}
\item There is a ring embedding $\phi: K \hookrightarrow End(V_0, \QQ)$ and an isomorphism of $\QQ$-vector spaces $f: V \rightarrow V_0$ such that for any $k \in K$ and $v \in V$, we have
$$f(kv) = \phi(k)(f(v))$$ 
\item There is a quadratic form $Q_0$ on $V_0$ such that $$Q_0(f(v), f(w)) = \Tr_{K/\QQ}(Q(v,w))$$
In particular, $Q_0(\phi(k)f(v), f(w)) = \Tr_{K/\QQ}(k \cdot Q(v,w))$
\end{enumerate} 

Conversely, given a Hodge structure of K3 type with real multiplication by $K$, one can find a unique $(V,Q)$.
\end{proposition}

In \cite{me1} the author introduced a construction called \say{corestriction} of quadratic spaces, which gives rise to a functor:
$$\Cor_{K/\QQ}: \{K \text{-quadratic spaces}\} \rightarrow \{\QQ \text{-quadratic spaces}\} $$
such that any Hodge structure of K3 type with RM is the corestriction of some $K$-quadratic spaces satisfying certain signature condition and discriminant condition.

We recall the construction of the corestriction of quadratic spaces functor: Let $K,d,V,n,Q$ be as stated above. Let $\sigma_1, ..., \sigma_d$ be the $d$ distinct embeddings of $K$ into $\RR$. First we consider the space $$V \otimes_{\QQ} \RR \cong (K^{\sigma_1})^n \oplus (K^{\sigma_2})^n \oplus ... (K^{\sigma_d})^n$$ where $K^{\sigma_i}$ denotes $\RR$ with $K$ embedded into it via $\sigma_i$. Under this isomorphism, a vector $kv = v \otimes k$ in $V$ will be mapped to $ (\sigma_1(k)v^{(1)}, ..., \sigma_d(k)v^{(d)})$.
This space is endowed with a semilinear automorphism by the Galois group $G = Gal(K/\QQ)$: if $\sigma_j = \sigma_i \circ g$ for a $g \in G$, then $g$ takes $\sigma_i(k)\otimes v^{(i)}$ to $\sigma_j(k) \otimes v^{(j)}$.

\begin{definition}
The corestriction of $V$ from $K$ to $\QQ$ is defined to be the $\QQ$-subspace in $V \otimes_{\QQ} \RR$ invariant under the semilinear automorphisms by $G$. i.e.
$$\Cor_{K/\QQ} (V) = (V \otimes_{\QQ} \RR)^G $$
Similarly, use $Q^{\sigma_i}$ to denote the $\RR$-quadratic form $Q^{\sigma_i}( r_1 \otimes \sigma_i(x), r_2 \otimes \sigma_i(y)) := r_1 \cdot r_2 \sigma_i (Q(x,y))$. The corestriciton of $Q$, denoted $Q_0$ is the quadratic form on $(K^{\sigma_1})^n \oplus (K^{\sigma_2})^n \oplus ... (K^{\sigma_d})^n$ given by $Q^{\sigma_1} \oplus Q^{\sigma_2} \oplus ... Q^{\sigma_d} $. The $\QQ$ quadratic space we obtain in this way is denoted $\Cor_{K/\QQ}(V, Q)$, called corestriction of the quadratic space $(V,Q)$. Likewise, if $\Lambda$ is a projective $\mathcal{O}_K$ lattice in $(V,Q)$, we can define the corestriction of $\Lambda$ and denote it as $\Cor_{K/\QQ}(\Lambda, Q)$.
\end{definition}

We now prove that this space has a compatible $K$-multiplication structure that satisfies the conditions imposed by Zarhin's proposition:

\begin{lemma}
\begin{enumerate}
\item $K$ admits a canonical embedding $\phi$ into $\End(\Cor_{K/\QQ}(V), \QQ)$. Moreover, $\phi(k) \in \End(\Cor_{K/\QQ}(V), \ZZ)$ for any $k \in \mathcal{O}_K$.
\item $Q_0$ takes coefficients in $\QQ$ on $\Cor_{K/\QQ}(V)$. Moreover, let $\phi$ be the canonical imbedding of $K$ into $\End(\Cor_{K/\QQ}(V), \QQ)$, $f$ the canonical isomorphims of $\QQ$-vector spaces $f: V \rightarrow \Cor_{K/\QQ}(V)$, then we also have $Q_0(\phi(k)f(v), f(w)) = \Tr_{K/\QQ}(k \cdot Q(v,w))$.
\end{enumerate}
The discriminant group of $Q_0$ has order $$ | \disc(Q_0) |= \disc(K/\QQ)^n \cdot \Norm_{K/\QQ}(\disc(Q)).$$
\end{lemma}
\begin{proof} Without loss of generality we assume $V$ is $1$-dimensional (i.e. $V = K\cdot e$). The correstriction map in this case recovers the classical setting of Minkowski embeddings:
$$ \begin{array}{cccc}

f: & K & \hookrightarrow & \Pi (\RR^{\sigma_i}) \\
 & ke & \mapsto & (k^{\sigma_1}e^{(1)}, ..., k^{\sigma_d}e^{(d)})\\
\end{array}
$$
If we assume that $K$ has a primitive element $a$, we can describe $\phi$ simply by describing $a$. In particular, we can let $\phi$ to be the map such that $$\phi(a)f(k  \cdot e) := f(ak  \cdot e)$$ It is not hard to see that with respect to the $\QQ$-basis $\{f(e), f(a  \cdot e), ..., f(a^{d-1}  \cdot e) \}$ of $\Cor_{K/\QQ}(V)$, $\phi$ maps $a$ to its rational canonical form. This completes the proof of the first part.

A high-brow way of proving the rationality of $Q_0$ is by noticing that $Q_0$ is actually invariant under the semilinear automorphisms by the Galois group.

For the purpose of computation, we will write out $Q_0$ for a quadratic form $Q(e,e) = c$ on the 1-dimensional $K$-vector space $V = Ke = \QQ(a)e$. If we adopt the basis in the previous lemma, a straighforward computaion will give

$$Q_0 = \left ( \begin{array}{cccc}
\Tr_{K/\QQ}(c) & \Tr_{K/\QQ}(ac) & ... & \Tr_{K/\QQ}(a^{d-1} c) \\
\Tr_{K/\QQ}(ac) & \Tr_{K/\QQ}(a^2 c) & ... & \Tr_{K/\QQ}(a^d c) \\
 & ... & ... & \\
\Tr_{K/\QQ}(a^{d-1} c) & \Tr_{K/\QQ}(a^d c) & ... & \Tr_{K/\QQ}(a^{2d-2} c)\\

\end{array} \right )
$$

To prove the relation between $Q_0$ and $Q$, we observe that since we defined $\phi$ such that it satisfies $\phi(k)f(v) = f(k\cdot v)$, and that by the way we set up $Q_0$, we always have $$Q_0(\phi(k)f(v), f(w)) = (k \cdot Q(v,w))^{\sigma_1} + (k \cdot Q(v,w))^{\sigma_2} + ... +(k \cdot Q(v,w))^{\sigma_d}$$
The right hand side gives exactly $\Tr_{K/\QQ}(k \cdot Q(v,w))$.
\end{proof}

\begin{remark}
It is not difficult to see that if $\Lambda$ is an $\mathcal{O}_K$-integral lattice in $V$ (namely for any two $v, w \in \Lambda$, $Q(v,w) \in \mathcal{O}_K$), then $f(\Lambda)$ is a $\ZZ$ lattice in $\Cor_{K/\QQ}(V)$. Moreover, with the above description, we can compute the \emph{discriminant form} of the lattice. 
\end{remark}

At this point, all that remains is to make sure an $\mathcal{O}_K$-lattice is taken to an even lattice that admits a primitive embedding into the $K3$ lattice with signature $(2, n\cdot d -2)$. This poses some conditions on the quadratic form $Q$.

We say the signature of $Q$ is $(p_1, q_1) ,..., (p_d, q_d)$ if the signature of each embedding $Q^{\sigma_i}$ is given by $(p_i, q_i)$. The weight condition on the vectors in the transcendental lattice imposes the following condition on $Q$:

\begin{lemma}\label{sig} The signature of $Q$ must be of the form $(2, n-2), (0,n), ..., (0,n)$.
\end{lemma}
\begin{proof} We consider where the unique (up to scalar multiplication) $(2,0)$ vector lives. By our construction, over $\CC$ it lives inside some specific embedding of $\sigma_{i, \CC}: V \hookrightarrow \RR^n \otimes \CC \subset \Cor_{K/\QQ}(V, Q)\otimes \CC$. WLOG suppose that $i = 1$. Then we claim that $(0,2)$ vector must live inside the same subspace. This is essentially because the $\RR$-forms of $SO(n)$ and $SO_n(\CC)$ shares the same representation, hence if we consider the complex conjugate of the $(2,0)$ vector, it would be a $(0,2)$ vector living inside $\sigma_{i,\CC}: V \hookrightarrow \RR^n \otimes \CC $. Since up to scalar multiplication there is only one $(0,2)$ vector, we know that the signature of $Q$ is exactly stated as the lemma.
\end{proof}

\begin{theorem}
Any Hodge structure of $K3$ type with real multiplication by a totally real field $K$ of degree $d$ can be obtained through the correstriction of some $K$-vector space $V$ of dimension $n$ equipped with a quadratic form $Q$ such that the following conditions are satisfied:

\begin{enumerate}
\item $Q$ has signature $(2, n-2), (0,n), ..., (0,n) $ and maps to an even quadratic form under $\Cor_{K/\QQ}$.
\item The Hodge struture can be made integral if there exists an $\mathcal{O}_K$ lattice $\Lambda \subset V$ such that $\Cor_{K/\QQ}(\Lambda)$ admits a primitive embedding into the $K3$ lattice. 
\end{enumerate}
\end{theorem}

\section{The case of $n=3$}\label{sec5}

In this section we prove the following correspondence:

$$\left (
\begin{array}{c}
\text{Abelian fourfold with QM} \\
\text{by $B_{\Cor}$, and Mumford-Tate group}\\
\text{preserving a rank 4 quadratic $\QQ$-form $Q_0$}
\end{array} \right ) \leftrightarrow \left (
\begin{array}{c}
\text{K3 surfaces $S$ with real multiplication} \\
\text{by the quadratic field $K$ such that}\\
\text{$Q_{0,K} \cong C \times \overline{C}$, $T(S) = \Cor_{K/\QQ}(Q)$}
\end{array} \right ) $$
where $\overline{C}$ denotes the Galois conjugation of $C$ over $\QQ$ and $Q$ denotes (up to scaling) the ternary quadratic form with the right signture that is determined by $C$.

\begin{theorem}Let S be a K3 surface with real multiplication by arbitrary RM field $K$ with degree no greater than $6$, and let its intersection form on the transcendental space $\Cor_{K/\QQ}(Q)$ be given by a ternary $K$-quadratic form $Q$. Then $Q$ is associated to a quaternion algebra $B$ defined over $K$ satisfying the ramifying condition at infinity (see \ref{ram}). Up to isogeny, the Kuga-Satake variety of $S$ is the products of Abelian varieties $A$ given by the generalized Mumford construction we introduced in section~\ref{sec3}. In particular, the endomorphism algebra of a simple factor is either trivial (when $B_{\Cor} = M_2(\QQ)$) or $B_{\Cor}$.
\end{theorem}
\begin{proof}
The proof of the theorem is mainly a generalization of the results in \cite{Gal00} and section 3 of \cite{me1}. Let $K, Q, B$ be as stated. Recall the classical fact that the norm-1 group $G \subset B$ can be realized as an (at most) double cover of $\SO_K(Q)$ (cf. Proposition 4.5.10 in \cite{Voi17}). Hence $G$ and the special Mumford-Tate group of $T(S)$ shares the same even weight representation. Furthermore, we can locate a copy of $T(S)_{\CC}$ in $B_{\Cor}\otimes \Cor_{K/\QQ}(B) \otimes \CC$ in the similar fashion as \cite{Gal00}: 
\begin{equation}\label{split}
T(S)_{\CC} = W_{2,0,0,..,0} \oplus W_{0,2,0,...,0} \oplus ... \oplus W_{0,0,...,0,2} \subset H^2(A, \CC). 
\end{equation}
And a $\QQ$ basis of $T(S)$ can be given by $\Id \otimes \Cor_{K/\QQ}(B^0) \subset B_{\Cor}\otimes \Cor_{K/\QQ}(B)$, where $B^0$ denotes the space of trace 0 quaternion elements. The lemma (lemma \ref{sig}) we have proved on the signature of $Q$ determines the ramifying data of $B$ at infinity, which is exactly the condition stated in ~\ref{ram}. 

Moreover, this construction can be made integral. Lattices can be obtained by looking at the intersection of quaternion orders and trace zero space. For specific computations, see section 3 of \cite{me1}.

In \cite{Gee06} van Geemen worked out the representation of the Kuga-Satake variety of RM K3 surfaces over $\RR$, which are direct sums of the representations given by $\Nmm$ (see lemma~\ref{ksrep}). Since $T(S)_{\CC}$ has Hodge numbers $(1,3d-2,1)$ if and only if one of the subspaces in (\ref{split}) is of Hodge number $(1,1,1)$ which is simultaneously the subspace in $\Cor_{K/\QQ}(B^0)$ where $Q^{\sigma_i}$ splits, this necessarily forces the Deligne's torus to be embedded in $G$ exactly the way we specified in section~\ref{sec3}. This implies that the Kuga-Satake variety of $T(S)$ is indeed given by the generalized Mumford construction.

\end{proof}

\begin{remark} This statement is also proven by Schlikewei in \cite{Sch09} by finding copies of $\Cor_{K/\QQ}(B)$ in $C^0(T(S))$. Still, one should note that the isomorphism of Hodge substructures in \cite{Sch09} is not canonical. The generalized Mumford's construction provides a canonical way to roundabout the choices made in \cite{Sch09}. In fact, for a general quadratic form $Q$ with signature $(2, n-2), (0,n), ... (0,n)$, we can replace $\Cor_{K/\QQ}(B)$ with $\Cor_{K/\QQ}(C^0(Q))$, and $B_{\Cor}$ with a minimal dimensional Brauer group representative of $\Cor_{K/\QQ}(C^0(Q))$, which by Merkurjev's theorem would become a quaternion over $\QQ$. Analogous computations to section \ref{sec3} would prove that the resulting abelian variety has the desired endomorphism, which is indeed Brauer-equivalent to Schlikewei's result.
\end{remark}

We now focus on the case when $d=2$ and $n=3$, and prove the following theorem:

\begin{theorem} Any simple Abelian fourfold $A$ with endomorphism by a definite quaternionic order can be realized as a Kuga-Satake variety of a K3 surface of Picard rank $16$ with real multiplication by a quadratic field, and the relation between the polarization of $A$ and the transcendental lattice of the K3 surface is exactly given by theorem \ref{main}. 
\end{theorem}

\begin{proof}To prove the theorem, we show that 1) there is a Hodge structure of K3 type in the second cohomology of $A$, and 2) the special Mumford-Tate group of $A$ is a (double cover of) $\SO_K(Q)$ for some ternary quadratic form $Q$ with signature $(2,1),(0,3)$.

It has been proven (see section 6.1 of \cite{Z-M95}) that for a simple Abelian fourfold with endomorphism by a definite quaternion order, the Lie algebra of the special Mumford-Tate group over $\CC$ is isomorphic to $\mathfrak{so}_4$. Moreover, there's a canonical isomorphism between the Lie algebra of $\SL_2 \times \SL_2$ and $\SO(4)$ (cf. pp 274-275 of \cite{F-H91}), and it is established that the standard representation of $\SO(4)$ corresponds to the $W_{1,1}$ of $\SL_2 \times \SL_2$. In this proof we shall adopt the $\SL_2$ notation. We have the following computation: $\Sym^2(H^1(A)) = 3\Sym^2 W_{1,1} \oplus \wedge^2 W_{1,1}$, and $\wedge^2(H^1(A)) = \Sym^2 W_{1,1} \oplus 3 \wedge^2 W_{1,1}$. By comparing the Hodge numbers of $\Sym^2(H^1(A))$ and $H^2(A)$ and counting the copies of sub-representations, it is not hard to work out the Hodge number of $\Sym^2 W_{1,1}$ and $\wedge^2 W_{1,1} = W_{2,0}\oplus W_{0,2}$, which are $(3,3,3)$ and $(1,4,1)$ respectively. Hence there is a Hodge structure of K3 type in $H^2(A \times A)$.

The identification of $\SL_2 \times \SL_2$ with $\SO(4)$ given by Fulton and Harris is geometric: it is obtained by the canonical isomorphism between a quadric surface in $\PP^3$ and the two rulings by conics. Hence in our case, the special Mumford-Tate group is a norm-1 subgroup of a quaternion $B$ satisfying the splitting condition if and only if the associated quadric hypersurface does not decompose into the product of conics over $\QQ$ but over a quadratic field extension, over which the two conics are conjugate to each other. This quadratic field extension, according to our configuration, is exactly the RM field of the K3 surface, and the corestriction of the conics will give the transcendental lattice of the K3 surface.

It remains to show that the quadric surface cannot split over $\QQ$ as the product of two conics $C_1$ and $C_2$. Suppose otherwise. Then we can read off the two factors of the $\QQ$-form of $\SL_2$ by associating each $C_i$ with a class in $\Br(\QQ)$ hence a quaternion algebra's norm-1 subgroup $G_i$. However, in this case the $\wedge^2 W_{1,1}$ decomposes since $W_{2,0}$ and $W_{0,2}$ are now $\QQ$-space. The resulting Abelian variety would then be products of Abelian varieties of lower dimension which are no longer simple.
\end{proof}

\end{document}